\documentclass{amsart}

\usepackage{amssymb,amsmath}
\usepackage[mathscr]{eucal}
\usepackage{color}

\newcommand\blfootnote[1]{%
  \begingroup
  \renewcommand\thefootnote{}\footnote{#1}%
  \addtocounter{footnote}{-1}%
  \endgroup
}

\newtheorem{theorem}{Theorem}[section]
\newtheorem{lemma}[theorem]{Lemma}
\newtheorem{proposition}[theorem]{Proposition}
\newtheorem{corollary}[theorem]{Corollary}

\theoremstyle{definition}
\newtheorem{definition}[theorem]{Definition}

\theoremstyle{remark}

\newcommand{\Rd}{\mathbb{R}^d}

\newcommand{\be}{\begin{equation}}
\newcommand{\ee}{\end{equation}}

\definecolor{darkred}{rgb}{0.8,0,0}

\begin{document}

\title[Dilation estimates for Wiener amalgam spaces of Orlicz type]
{Dilation estimates for Wiener amalgam spaces of Orlicz type}

\author{B\"{u}\c{s}ra Ar\i s}
\address{\hspace{-\parindent} B\"{u}\c{s}ra Ar\i s, \.{I}stanbul University, Faculty of Science, Department of Mathematics, 34134 Veznec$\dot{\hbox{\i}}$ler, \.{I}stanbul, Turkey}
\email{busra.aris@istanbul.edu.tr}

\author{Nenad Teofanov$^*$}\blfootnote{$^*$ corresponding author}
\address{\hspace{-\parindent} Nenad Teofanov, University of Novi Sad, Faculty of Science, Department of Mathematics and Informatics, Novi Sad, Serbia
}
\email{tnenad@dmi.uns.ac.rs}

\author{Serap \"Oztop}
\address{\hspace{-\parindent} Serap \"Oztop, \.{I}stanbul University, Faculty of Science, Department of Mathematics, 34134 Veznec$\dot{\hbox{\i}}$ler, \.{I}stanbul, Turkey}
\email{oztops@istanbul.edu.tr}

\subjclass[2020]{}
\keywords{Wiener amalgam spaces, Orlicz spaces, dilation operator}

\begin{abstract}
We extend dilation properties of Wiener amalgam spaces when the local and global componenets are Lebesgue spaces to  a more general setting of Orlicz spaces. We recover the result of Cordero and Nicola
when restricted to Lebesgue spaces. In addition, we prove continuity of the Zak transform on  Wiener amalgam spaces with Orlicz spaces as their local components.
\end{abstract}
\maketitle

\section{Introduction} \label{sec1}

Wiener amalgam spaces are a family of function spaces which amalgamates local properties with global ones. This idea can be traced back to the work of N. Wiener in his theory of generalized harmonic analysis. He considered the amalgam spaces
$$
W(L^1 , L^2), \quad W(L^2 , L^1), \quad  W(L^1 , L^{\infty}), \quad \text{and}
\quad  W(L^{\infty} , L^1)
$$
on the real line in \cite{Wie1, Wie2, Wie3}, where the so-called standard amalgam space $W(L^p , L^q)$, $ 1 \leq p,q \leq \infty$,  is defined by the norm
\begin{equation}\label{disknorm}
\|f\|_{W(L^p , L^q)}=\Big(\sum_{n \in \mathbb{Z}} \Big(\int_{n}^{n+1} |f(t)|^p dt \Big)^{\frac{q}{p}} \Big)^{\frac{1}{q}},
\end{equation}
with usual modification when $p$ or $q$ is infinity. Thus, $L^p $ is the local, and $ L^q$ is the global component of $W(L^p , L^q)$.

In \cite{Fei2, Fei}, H. G. Feichtinger introduced a generalization of Wiener amalgams to the wide range of Banach spaces of functions as their local and global components. Wiener amalgam spaces are mostly studied for Lebesgue spaces on the real line. C. Heil considered weighted Wiener amalgam spaces $W(L^p, L_{\omega}^q)$ on locally compact groups and on the real line  in \cite{ Heil1, Heil, Heil2}.

Wiener amalgams find themselves as a very useful tool, for example, in sampling theory \cite{Heil2} and in  time-frequency analysis \cite{Groc}. It turned out that continuity properties of certain operators can be conveniently described in the context of Wiener amalgam spaces (see \cite{Cor-Nic, Cor-Rod}).

In contrast to the most familiar setting of  Lebesgue spaces, we consider  Wiener amalgams with Orlicz spaces as their local and global components. Beside the fact that Orlicz spaces generalize $L^p$ spaces, they include other important classes, such as the Zygmund space $L \log^{+} L$ which is a Banach space related to Hardy-Littlewood maximal functions. Orlicz spaces also contain certain Sobolev spaces as their subspaces.

In \cite{ar-oz}, Arıs and \"Oztop considered Wiener amalgam spaces with respect to Orlicz spaces $W(L^{\Phi}(\mathbb{A}), L^{1}(\mathbb{A}))$ and $W(L^{\infty}(\mathbb{A}), L^{\Phi}(\mathbb{A}))$ on the affine group $\mathbb{A}$. They gave some properties of Wiener amalgam spaces of Orlicz type and proved convolution relations for $W(L^{\Phi}(\mathbb{A}), L^{1}(\mathbb{A}))$ and $W(L^{\infty}(\mathbb{A}), L^{\Phi}(\mathbb{A}))$. In \cite{ar-oz2} they generalized these spaces to the weighted Orlicz spaces on locally compact groups, which they called the Orlicz amalgam spaces, and studied some basic properties such as translation invariance, density, duality. Moreover, it is shown that these spaces are Banach algebras under convolution multiplication.

In this paper we give dilation estimates for Wiener amalgam spaces of Orlicz type. Such estimates were derived in \cite{Cor-Nic, Cor-Rod} in the context of Lebesgue spaces.  We recover these results since they arise as a special case of Orlicz type spaces considered in this paper. To the best of our knowledge, dilation properties of Wiener amalgam spaces of Orlicz type have not been studied before.

In addition, we prove the continuity of the Zak transform on Wiener amalgam spaces when the local component is an Orlicz space. The Zak transform is an essential tool when proving the well-known
"Amalgam Balian-Low Theorem", which we recall in Section \ref{sec5}.
In fact, the Zak transform gives an information whether a Gabor system is an orthonormal basis or a Riesz basis for $ L^2 (\Rd) $, see Lemma \ref{lem.11.9.4}. It turns out that the Gabor frames (i.e. Gabor systems which are  frames, see Section \ref{sec5}) give rise to (Gabor) expansions which are convergent in the entire range of weighted amalgam spaces, see \cite{Bckor, GrHeOk, K-O}. Gabor frame expansions in the context of Wiener amalgam spaces of Orlicz type which extend results from \cite{GrHeOk} will be the subject of our future study.

To end this introduction, we note that  Orlicz spaces were recently used in the context of Orlicz modulation spaces, see \cite{Grtu, Tumo}. It turned out that  Orlicz modulation spaces provide strictly sharper estimates for pseudo-differential operators when compared to the classical modulation spaces. Since the Fourier transform images of modulation spaces are certain Wiener amalgam spaces, it is natural to expect that Wiener amalgam spaces with respect to Orlicz spaces, which are studied in this paper, can be used in the context of pseudo-differential operators as well.

This paper is organized as follows. In Section \ref{sec2}, we review necessary background on Orlicz spaces. In Section  \ref{sec3}, we give basic structure of Wiener amalgam spaces of Orlicz type on $\Rd$ which we denote by $W(L^{\Phi_1} (\Rd), L^{\Phi_2} (\Rd))$. In Section  \ref{sec4}, we study dilation properties of Wiener amalgam spaces of Orlicz type (Lemma \ref{lem.2.5.2} and Proposition \ref{main.result} ). In Section
 \ref{sec5}, we extend a result of Heil
( \cite[Proposition 11.9.5]{Heil})
to Orlicz spaces as local components, and recall the famous "Amalgam Balian-Low Theorem".

\section{Prelimaniries on Orlicz spaces} \label{sec2}

Let us recall some facts concerning Young functions and Orlicz spaces. We refer to \cite{R-R, R-R2} for proofs and details.

A function $\Phi: [0,\infty) \to [0, \infty]$ is called {\em a Young function} if $\Phi$ is convex, $\Phi(0)=0$ and $\lim \limits_{x \to \infty} \Phi(x)=\infty$. For a Young function $\Phi$, $\Phi^{-1}$ is defined by
$$
\Phi^{-1}(y)=\inf\{x>0: \Phi(x)>y\}, \qquad y \geq 0,
$$
where $\inf \emptyset =\infty$. For example, if $\Phi(x)=x^p$, $p \geq 1$, $x \geq  0$,
then $\Phi^{-1}(y)=y^{1/p}$, $ y \geq 0$.

For a Young function $\Phi$, {\em the complementary function} $\Psi$ of $\Phi$ is given by
$$\Psi(y)= \sup\{xy- \Phi(x) : x >0\}, \qquad y \geq 0.$$
Then $\Psi$ is also a Young function, and $(\Phi, \Psi)$ is called a {\em complementary Young pair}. Also, if $\Psi$ is the complementary function of $\Phi$, then $\Phi$ is the complementary  function of $\Psi$, so $(\Psi, \Phi)$ is a complementary Young pair as well. For a given Young function $\Phi$ and its complementary function $\Psi$, we have the Young inequality
$$xy \leq \Phi(x) + \Psi(y), \qquad x,y \geq 0.$$

Let us note that in general, there is a straightforward method to construct various complementary pairs of strictly increasing continuous Young functions as described in \cite[Theorem 1.3.3]{R-R}. Suppose that $\varphi: [0,\infty) \to [0,\infty)$ is a continuous strictly increasing function with $\varphi(0) = 0$ and $\lim_{x \to \infty}\varphi(x)=\infty$. Then
$$\Phi(x)=\int_{0}^{x} \varphi(y)dy$$
is a continuous strictly increasing Young function and
$$\Psi(x)=\int_{0}^{y} \varphi^{-1}(x)dx$$
is the complementary Young function of $\Phi$ which is also continuous and strictly increasing. Here $\varphi^{-1}$ is the inverse function of $\varphi$. There are some examples of complementary Young pairs satisfying the above construction (see \cite[p. 15]{R-R}).
\begin{enumerate}
\item If $\Phi(x)=x^p$, then $\Psi(x)=x^q$ for $1<p,q<\infty$ with $\frac 1p +\frac 1q=1$,
\item If $\Phi(x)=x \ln(1+x)$, then $\Psi(x)\asymp \cosh(x)-1$,
\item If $\Phi(x)= \cosh(x)-1$, then $\Psi(x)\asymp x \ln(1+x)$.
\end{enumerate}

By the definition, a Young function can attain the value $\infty$ at a point. When considering the pair of complementary Young functions $(\Phi, \Psi)$ we often assume that $\Phi$ is continuous on $[0,\infty)$ and increasing on $(0,\infty)$. Note that even is $\Phi$ is a continuous function, it may happen that $\Psi$ is not continuous.

Let $(\Phi_1, \Psi_1)$ and $(\Phi_1, \Psi_2)$ be  complementary Young pairs. If $\Phi_1(x) \leq \Phi_2 (x)$ for all $x \geq x_0 \geq 0$, then we have $\Psi_2(y)\leq \Psi_1(y)$ for all $y \geq y_0 =\Phi_1(x_0) \geq 0$.

Let $\Phi_1, \Phi_2$ be two Young functions. If there exist a constant $c>0$ and $x_0 \geq 0$ (depending on $c$) such that $\Phi_1 (x) \leq \Phi_2(cx)$ for all $x \geq x_0$, then we say that $\Phi_2$ is
{\em stronger than} $\Phi_1$ and denote this by $\Phi_1 \prec \Phi_2$.
If, in addition, $ c=1 $ and $x_0 =0$, we say that $\Phi_2$ is {\em strictly stronger than} $\Phi_1$.
If $\Phi_1 \prec \Phi_2$ and $\Phi_2 \prec \Phi_1$, then we write $\Phi_1 \asymp \Phi_2$
and we may treat them as equivalent functions. Also, $\Phi_1 \prec  \Phi_2$ if and only if $\Phi_2^{-1} (y) \leq c~ \Phi_1^{-1} (y)$ for all $y \geq y_0 =\Phi_1(x_0)$.

 Given a Young function $\Phi$, we define
 $$\mathcal{L}^{\Phi}(\Rd)=\Big\{f : \Rd \to \mathbb{C} ~\mbox{measurable} : \int_{\Rd} \Phi( |f(x)|) dx < \infty\Big\}.$$
In general $\mathcal{L}^{\Phi}(\Rd)$ is not a linear space. For that reason we define
{\em the Orlicz space}  $L^{\Phi}(\Rd) $ as follows:
\begin{equation}\label{orliczdefn}
L^{\Phi}(\Rd)= \Big\{f : \Rd \to \mathbb{C} ~\mbox{measurable} : \int_{\Rd} \Phi( \alpha |f(x)|) dx < \infty ~\mbox{for some} ~\alpha >0 \Big\}.
\end{equation}
As usual, $f$ is a representative of the corresponding equivalence class of measurable functions. $ L^{\Phi}(\Rd)$  is a Banach space under  {\em the Orlicz norm} $\|\cdot\|_{L^{\Phi}}$ defined for $f \in L^{\Phi}(\Rd)$ by
$$\|f\|_{L^{\Phi}} = \sup \Big\{ \int_{\Rd} |f(x)g(x)|dx : \int_{\Rd} \Psi(|g(x)|) dx \leq 1 \Big\},$$
where $\Psi$ is the complementary Young function of $\Phi$.

 {\em The Luxemburg norm} $\|\cdot\|_{L^{\Phi}}^{\circ}$ on $L^{\Phi}(\Rd)$ is defined to be
$$\|f\|_{L^{\Phi}}^{\circ} =\inf \Big\{k>0 : \int_{\Rd} \Phi\bigg(\frac{|f(x)|}{k}\bigg) dx \leq 1\Big \}.$$
It is known that these norms are equivalent, that is,
\begin{equation}\label{eq.norm}
\|\cdot\|_{L^{\Phi}}^{\circ} \leq \|\cdot\|_{L^{\Phi}} \leq 2 \|\cdot\|_{L^{\Phi}}^{\circ},
\end{equation}
and
$$\|f\|_{L^{\Phi}}^{\circ} \leq 1 ~\mbox{if and only if}~ \int_{\Rd} \Phi(|f(x)|) d x \leq 1.$$

We denote norm equivalence of Banach spaces $ (X, \| \cdot \|_X ) $ and
$ (Y, \| \cdot \|_Y ) $ by $ \| \cdot \|_X   \asymp \| \cdot \|_Y$ or
$ X \asymp Y$. Although the same notation is used for the equivalence of Young functions,
the meaning will be clear from the context.

A Young function $\Phi$ satisfies {\em the $\Delta_2$ condition} if there exist a constant $K>0$ and an $x_0 \geq 0$ such that $\Phi(2x) \leq K \Phi(x)$ for all $x \geq x_0$. In this case, we write $\Phi \in \Delta_2$. Considering $\Phi(x)=x^p$, $p \geq 1$, one see that $\Phi \in \Delta_2$ for $K \geq 2$. On the other hand, if $\Phi(x)=e^x -1$, then $\Phi$ does not satisfy the $\Delta_2$ condition. The  $\Delta_2$ condition plays a central role in the structure theory of Orlicz spaces.

Let $\mathcal{S}^{\Phi}(\Rd)$ be the closure of the linear space of all step functions in $L^{\Phi}(\Rd)$. Then $\mathcal{S}^{\Phi}(\Rd)$ is a Banach space which contains $C_c (\Rd)$, the space of all continuous functions on $\Rd$ with compact support, as a dense subspace \cite[Proposition 3.4.3]{R-R}. Moreover, if $(\Phi, \Psi)$ is a complementary Young pair, then  $(\mathcal{S}^{\Phi}(\Rd) )^*$, the dual of $\mathcal{S}^{\Phi}(\Rd)$, can be identified with $L^{\Psi}(\Rd)$ in a natural way \cite[Theorem 4.1.6]{R-R}. Another useful characterization of $\mathcal{S}^{\Phi}(\Rd)$ is that $f \in \mathcal{S}^{\Phi}(\Rd)$ if and only if for every $\alpha >0$, $\alpha f \in \mathcal{S}^{\Phi}(\Rd)$. If $\Phi \in \Delta_2$, then it follows that $\mathcal{S}^{\Phi}(\Rd)= L^{\Phi}(\Rd)$ so that $L^{\Phi}(\Rd)^*=L^{\Psi}(\Rd)$. If, in addition, $\Psi \in \Delta_2$, then the Orlicz space $L^{\Phi}(\Rd)$ is a reflexive Banach space.



We  also recall H{\"{o}}lder's inequality which states that if $f \in L^\Phi(\Rd)$ and $g\in L^\Psi(\Rd)$, then  $fg \in L^1(\Rd)$ and
$$
\|fg\|_{L^1} \leq 2\|f\|_{L^\Phi}^\circ \|g\|_{L^\Psi}^\circ.
$$

A normed space $(Y, \| \cdot\|_{Y})$ consisting of measurable of complex valued functions on a measurable space $X$ is called {\em  solid} if for each measurable $f: X \to \mathbb{C}$ satisfying $|f| \leq |g|$ almost everywhere for some $g \in Y$, we have $f \in Y$ and $\|f\|_{Y} \leq \|g\|_{Y}$. Since the Young function $\Phi$ is increasing, the Orlicz space $L^{\Phi}(\Rd)$ is a solid space, \cite[Theorem 2]{R-R}. Also, if the right derivative of a Young function $\Phi$ at zero is positive, i.e., ${\Phi'}_{+}(0)>0$, then the inclusion $L^{\Phi}(\Rd) \subseteq L^{1}(\Rd)$ is valid. This implies that there exists a constant $c>0$ such that
\begin{equation}\label{inc}
\|f\|_{L^{1}} \leq c \|f\|_{L^{\Phi}},
\end{equation}
for every $f \in L^{\Phi}(\Rd)$ \cite[Proposition 1]{R-R2}.

 For $1 \leq p < \infty$ and the Young function $\Phi(x) =x^p$, the space $L^{\Phi} (\Rd)$ is the classical Lebesgue space $L^{p}(\Rd)$ and the norm $\|\cdot\|_{L^{\Phi}}$ is equivalent to the usual Lebesgue spaces norm $\|\cdot\|_{L^{p}}$.

If $p=1$, then we obtain the space $L^{1}(\Rd)$. In this case, the complementary Young function of $\Phi(x) =x$ is
\begin{equation}\label{yf}
\Psi(x)=
\begin{cases}
0, & 0 \leq x \leq 1,\\
\infty,& x>1,
\end{cases}
\end{equation}
and $\|f\|_{L^{\Phi}} =\|f\|_{L^{1}}$ for all $f \in L^{1}(\Rd)$. If $p=\infty$, then for the Young function $\Psi$ given in \eqref{yf}, the space $L^{\Psi} (\Rd)$ is equal to the space $L^{\infty}(\Rd)$ and we have $\|f\|_{L^{\Psi}} =\|f\|_{L^{\infty}}$ for all $f \in L^{\infty}(\Rd)$.

In addition, for the Young function
\begin{equation}\label{yf1}
\Phi_s(x)=
\begin{cases}
x, & 0 \leq x \leq 1,\\
\infty,& x>1,
\end{cases}
\end{equation}
the space $L^{\Phi_s} (\Rd)$ becomes $L^1 (\Rd) \cap L^{\infty}(\Rd)$ with the norm $$\|x\|_{L^{\Phi_s}}=\max\{\|x\|_{L^1}, \|x\|_{L^{\infty}}\}.$$
We note that $$ $$
\begin{equation*}
\Phi_s(x) ^{-1} =
\begin{cases}
x, & 0 \leq x \leq 1,\\
1,& x>1,
\end{cases}
\end{equation*}
If we consider the Young function
\begin{equation}\label{yf2}
\Phi_b(x)=
\begin{cases}
0, & 0 \leq x \leq 1,\\
x-1,& x>1,
\end{cases}
\end{equation}
we obtain the space $L^1 (\Rd) + L^{\infty}(\Rd)$ with the norm
$$
\|x\|_{L^{\Phi_b}}=\inf\limits_{x=x_1 +x_2}(\|x_1\|_{L^1}+ \|x_2\|_{L^{\infty}}).
$$
We have $\Phi_b(x) ^{-1} = x+1$, $x \geq 0$. In addition,
$\Phi_s$ and $\Phi_b$ are complementary Young functions \cite[p. 52]{Maligranda} which satisfy
the  $\Delta_2$ condition.

Since
$$
L^1 (\Rd) \cap L^{\infty}(\Rd) \subset L^{\Phi}(\Rd) \subset L^1 (\Rd) + L^{\infty}(\Rd)
$$
for any Young function $\Phi$, we may consider $L^1 (\Rd) \cap L^{\infty}(\Rd)$ as the smallest Orlicz space, and  $L^1 (\Rd) + L^{\infty}(\Rd)$ as the biggest one \cite[p.100]{Maligranda}. We denote the spaces $L^1 (\Rd) \cap L^{\infty}(\Rd)$ and $L^1 (\Rd) + L^{\infty}(\Rd)$ by $L^{\Phi_s} (\Rd)$ and  $L^{\Phi_b} (\Rd)$, respectively. Also, we have $L^{\Phi_s}(K) \asymp L^{\infty}(K)$ for any compact subset $K \subset \Rd$.

We denote the translation by $T_y f(x)=f(x-y)$, the modulation by $M_y f(x)= f(x) e^{2 \pi i xy}$ and the dilation operators by $D_\lambda f(x)=f(\lambda x)$ for $x,y \in \Rd$ and $\lambda>0$. All of these operators are well-defined, linear and bounded operators on Orlicz spaces \cite{R-R}.

Properties of the dilation operator $D_{\lambda}$ when acting on Orlicz spaces are recently studied by
Blasco and Osan\c{c}l\i ol  in \cite{Os-Blasco}.
To formulate their result given in Lemma  \ref{cor.dil} below we need some preparation.

Given $\lambda >0$, another norm on the Orlicz space $L^{\Phi} (\mathbb{R})$ is defined by
$$\|f\|_{L^{\Phi}}^{\circ, \lambda} =\inf \Big\{k>0 : \int_{\mathbb{R}} \Phi\bigg(\frac{|f(x)|}{k}\bigg) dx \leq  \lambda\Big \}.$$

When $\lambda=1$ we obtain $\|f\|_{L^{\Phi}}^{\circ, 1} =\|f\|_{L^{\Phi}}^{\circ}$,  and for
$\lambda >0$ we have
\begin{equation}\label{eq.dil2}
\|D_{\lambda}f\|_{L^{\Phi}}^{\circ}= \|f\|_{L^{\Phi}}^{\circ, \lambda}, \qquad
\forall f \in L^{\Phi}(\Rd).
\end{equation}

By convexity, it can be easily seen that
$$
\frac{\lambda_1}{\lambda_2} \|\cdot\|_{L^{\Phi}}^{\circ, \lambda_1} \leq \|\cdot\|_{L^{\Phi}}^{\circ, \lambda_2} \leq \|\cdot\|_{L^{\Phi}}^{\circ, \lambda_1}
\;\; \text{when} \;\;  0 < \lambda_1 < \lambda_2.
$$
Therefore,
\begin{equation}\label{norm2}
 \lambda \|f\|_{L^{\Phi}}^{\circ, \lambda}\leq \|f\|_{L^{\Phi}}^{\circ} \leq  \|f\|_{L^{\Phi}}^{\circ, \lambda} \quad 0 <\lambda \leq 1,
\end{equation}
and
\begin{equation}\label{norm1}
 \|f\|_{L^{\Phi}}^{\circ, \lambda} \leq \|f\|_{L^{\Phi}}^{\circ} \leq \lambda \|f\|_{L^{\Phi}}^{\circ, \lambda}, \quad \lambda \geq 1 .
\end{equation}
On the other hand, by \eqref{eq.norm}, it follows that
\begin{equation}\label{eq.norm2.1}
 \lambda \|f\|_{L^{\Phi}}^{\circ, \lambda} \leq \|f\|_{L^{\Phi}} \leq 2  \|f\|_{L^{\Phi}}^{\circ, \lambda}, \quad 0 <\lambda \leq 1,
\end{equation}
and
\begin{equation}\label{eq.norm2.2}
 \|f\|_{L^{\Phi}}^{\circ, \lambda} \leq \|f\|_{L^{\Phi}} \leq 2\lambda \|f\|_{L^{\Phi}}^{\circ, \lambda}, \quad \lambda \geq 1.
\end{equation}

Let
$$
C_\Phi (\lambda):=\|D_{\lambda} \|_{L^{\Phi} \to L^{\Phi}}.
$$
Then $C_\Phi (\lambda)$ is non increasing, submultiplicative and $C_\Phi(1)=1$.

If $\Phi(x) =x^p$, $p \geq 1$, then we have $C_\Phi (\lambda)={\lambda}^{-1/p}$ for $L^p$-spaces.

The following lemma
gives an estimate for the norm of the dilation operator $D_{\lambda}$.
It  will be used when proving the main result of this paper.

\begin{lemma}\label{cor.dil}\cite{Os-Blasco}
    Let $\Phi$ be a Young function. Then,
    \begin{enumerate}
        \item[i)] $C_\Phi (\lambda) \geq \sup \limits_{ \mu >0} \frac{\Phi^{-1}(\mu)}{\Phi^{-1}(\lambda \mu)}$.
       \item[ii)] If $\Phi(st) \leq \Phi(s)\Phi(t)$ for all $s,t \geq 0$, then $C_\Phi (\lambda) \leq \frac{1}{\Phi^{-1}(\lambda)}$.
    \end{enumerate}
In particular, if $\Phi$ is submultiplicative and $\Phi(1)=1$, then $C_\Phi (\lambda)=\frac{1}{\Phi^{-1}(\lambda)}$.
\end{lemma}

Note that by \eqref{eq.dil2} and Lemma \ref{cor.dil}, we know that
$$
\|f\|_{L^{\Phi}}^{\circ, \lambda} =\|f_\lambda\|_{L^{\Phi}}^{\circ} \leq \frac{1}{\Phi^{-1}(\lambda)} \|f\|_{L^{\Phi}}^{\circ}, \;\;\; f \in L^{\Phi}(\mathbb{R}), \;\; \lambda>0.
$$
As usual, for $f$ defined on $\Rd$ and $\lambda>0$, we have
\begin{equation}\label{eq.dil}
\|f_\lambda\|_{L^{\Phi}}^{\circ} \leq  \frac{1}{\Phi^{-1}(\lambda^d)} \|f\|_{L^{\Phi}}^{\circ}.
\end{equation}

\section{Wiener-amalgam spaces of Orlicz type} \label{sec3}

Wiener amalgam spaces $W(B,C)$ were introduced by Feichtinger in \cite{Fei}. Here, the local component $B$ and the global component $C$ are solid and translation invariant Banach function spaces, and $B$ satisfies additional technical conditions. In \cite{Fei2, Fei}, one can find essential results on their embeddings, interpolation, convolution etc.  The most familiar examples of Wiener amalgam spaces are obtained for $B=L^p$, $C=L^q$, $1 \leq p, q \leq \infty$ (see \cite{Heil, Heil2, Maria}).

Wiener amalgam spaces with respect to Orlicz spaces are  recently considered by Arıs and \"{O}ztop. They are denoted by $W(L^{\Phi_1}(\Rd), L^{\Phi_2}(\Rd))$ and consist of functions that are locally in $L^{\Phi_1}(\Rd)$ and globally in $L^{\Phi_2}(\Rd)$. These spaces are called Orlicz amalgam spaces and their basic properties, such as translation invariance, inclusion relations, density and can be found in \cite{ar-oz, ar-oz2}. In addition, Arıs and \"{O}ztop considered  $W(L^{\Phi_1}(\Rd), L^{\Phi_2}(\Rd))$ as Banach algebras
under the convolution by using appropriate discrete type norms.

We note that $L^{\Phi_1}(\Rd)$ satisfies conditions for local components as given in \cite{Fei} (see also \cite{Heil2}).

Next we summarize some technical results from
\cite{ar-oz2} which  will be used in the next section.

Let $Q$ be a fixed compact subset of $\Rd$ with nonempty interior, and let
$\chi_{_{\scriptstyle Q+x}}$ denote the characteristic function of the set $ Q +x,$
$x \in \Rd$. Then the control function $F_f$ of a measurable function $f: \Rd \to \mathbb{C}$ is given by
$$
F_f (x)=F_{f}^{Q}(x)=\|f \chi_{_{\scriptstyle Q+x}}\|_{L^{\Phi_1}}, \quad x \in \Rd.
$$

\begin{definition}\label{defn.}
Let $\Phi_1, \Phi_2$ be Young functions. The Wiener amalgam space of Orlicz type, $W(L^{\Phi_1}(\Rd),L^{\Phi_2}(\Rd))$ consists of all measurable functions $f: \Rd \to \mathbb{C}$ such that $f \chi_{_{\scriptstyle Q+x}} \in L^{\Phi_1}(\Rd)$ for each $x \in \Rd$ and $F_{f}^{Q} \in L^{\Phi_2}(\Rd)$.

The Orlicz amalgam norm on $W(L^{\Phi_1}(\Rd),L^{\Phi_2}(\Rd))$ is defined by
	\begin{equation}\label{eq.amalgam}
	\|f\|_{W(L^{\Phi_1}, L^{\Phi_2})}:= \|F_f\|_{L^{\Phi_2}}=\big\| \|f \chi_{_{\scriptstyle Q+x}}\|_{L^{\Phi_1}} \big\|_{L^{\Phi_2}}.
	\end{equation}

	\end{definition}

In a similar way as for Orlicz spaces, on $W(L^{\Phi_1}(\Rd), L^{\Phi_2}(\Rd))$ we define the Luxemburg norm $\|\cdot\|_{W(L^{\Phi_1}, L^{\Phi_2})}^{\circ}$  by
$$
\|f\|_{W(L^{\Phi_1}, L^{\Phi_2})}^{\circ}=\big\| \|f \chi_{_{\scriptstyle Q+x}}\|_{L^{\Phi_1}}^{\circ} \big\|_{L^{\Phi_2}}^{\circ}.
$$
By \eqref{eq.norm}, we have the norm equivalence
			\begin{equation}\label{normeq}
			\|f\|_{W(L^{\Phi_1}, L^{\Phi_2})}^{\circ} \leq \|f\|_{W(L^{\Phi_1}, L^{\Phi_2})} \leq 4 \|f\|_{W(L^{\Phi_1}, L^{\Phi_2})}^{\circ}.
			\end{equation}
		
We have already mentioned that Orlicz spaces are generalizations of Lebesgue spaces. If we consider the Young functions $\Phi_1(x)=x^{p}$, $\Phi_2(x)=x^{q}$, $1 \leq p , q \leq \infty$ in Definition \ref{defn.}, then the  Wiener amalgam space of Orlicz type $W(L^{\Phi_1}(\Rd), L^{\Phi_2}(\Rd))$ reduces the classical Wiener amalgam space $W(L^{p}(\Rd), L^{q}(\Rd))$.

Since the Orlicz spaces $L^{\Phi_1}(\Rd)$ and $L^{\Phi_2}(\Rd)$ are solid and translation invariant, the Orlicz amalgam space $W(L^{\Phi_1}(\Rd), L^{\Phi_2}(\Rd))$ is a Banach space and its definition is independent of the choice of $Q$, in the sense that each choice of $Q$ yields the same space under an equivalent norm. We refer to \cite{ar-oz2} for the proofs of Propositions \ref{prop3.3} -- \ref{obsr.}.

\begin{proposition} \label{prop3.3}
	Let $\Phi_1, \Phi_2$ be Young functions and let $f \in W(L^{\Phi_1}(\Rd), L^{\Phi_2}(\Rd))$. Then the following hold:
	\begin{enumerate}
		\item[i)] $W(L^{\Phi_1}(\Rd), L^{\Phi_2}(\Rd))$ is a translation invariant space. That is, $T_x f \in W(L^{\Phi_1}(\Rd), L^{\Phi_2}(\Rd))$, $x \in \Rd$, and
  $$\|T_x f \|_{W(L^{\Phi_1}, L^{\Phi_2})} =  \|f\|_{W(L^{\Phi_1}, L^{\Phi_2})}.$$
		
		\item[ii)] If $\Phi_1, \Phi_2 \in \Delta_2$, then the mapping $x \mapsto T_x f$ from $\Rd$ into $W(L^{\Phi_1}(\Rd), L^{\Phi_2}(\Rd))$ is continuous.
	\end{enumerate}
\end{proposition}

\begin{proposition}\label{thm6}
	Let $(\Phi, \Psi)$ be a complementary Young pair with ${\Phi'}_{+}(0)>0$ and ${\Psi'}_{+}(0)>0$. Then,
	$$W(L^{\Phi}(\Rd), L^{\Phi}(\Rd))=L^{\Phi}(\Rd).$$
\end{proposition}

\begin{proposition}\label{thm4}
	Let $\Phi, \Phi_1 , \Phi_2$ be Young functions. If $\Phi_1 \prec \Phi_2$, then $$W(L^{\Phi_2}(\Rd), L^{\Phi}(\Rd)) \subseteq W(L^{\Phi_1}(\Rd), L^{\Phi}(\Rd)),$$
 and $\|f\|_{W(L^{\Phi_1}, L^{\Phi})} \lesssim \|f\|_{W(L^{\Phi_2}, L^{\Phi})}$.

 If, in addition, $\Phi \asymp \Phi_2$, then
 $$W(L^{\Phi_1}(\Rd), L^{\Phi}(\Rd)) = W(L^{\Phi_2}(\Rd), L^{\Phi}(\Rd)),$$
and $\|f\|_{W(L^{\Phi_1}, L^{\Phi})} \asymp \|f\|_{W(L^{\Phi_2}, L^{\Phi})}$.
\end{proposition}

The H\"{o}lder inequality for $W(L^{\Phi_1}(\Rd), L^{\Phi_2}(\Rd))$ can be formulated as follows.

	Let $(\Phi_1, \Psi_1), (\Phi_2, \Psi_2)$ be complementary Young pairs. Then
	$$\|fg\|_{W(L^1, L^1)}^{\circ} \leq 4 \|f\|_{W(L^{\Phi_1}, L^{\Phi_2})}^{\circ} \|g\|_{W(L^{\Psi_1}, L^{\Psi_2})}^{o},$$
	for all $f \in W(L^{\Phi_1}(\Rd), L^{\Phi_2}(\Rd))$ and $g \in W(L^{\Psi_1}(\Rd), L^{\Psi_2}(\Rd))$.

\begin{proposition}\label{dual}
    Let $(\Phi_1, \Psi_1)$, $(\Phi_2, \Psi_2)$ be complementary Young pairs. If $\Phi_1, \Phi_2 \in \Delta_2$, then the dual space of $W(L^{\Phi_1}(\Rd), L^{\Phi_2}(\Rd))$ is $W(L^{\Psi_1}(\Rd), L^{\Psi_2}(\Rd))$.
\end{proposition}

Finally, we recall an equivalent discrete-type norm on the  Wiener amalgam space of Orlicz type.

\begin{proposition}\label{obsr.}
	Let  $Q =[0,1]^d$ and let $\Phi_1, \Phi_2$ be any Young functions. Then, we have
	\begin{equation}\label{denk}
	\|f\|_{W(L^{\Phi_1}, L^{\Phi_2})} \asymp \big\|\big(\|f \chi_{_{\scriptstyle Q+k}} \|_{L^{\Phi_1}} \big)_{k \in \mathbb{Z}^d} \big\|_{{\ell}^{\Phi_2}}.
	\end{equation}
 Here, ${\ell}^{\Phi_2}$ is the Orlicz sequence space which is given by \cite{L-T, Rao}.
	\end{proposition}

\section{Main Results} \label{sec4}

In this section, we study dilation properties of Wiener amalgam spaces of Orlicz type
$W(L^{\Phi_1}(\Rd), L^{\Phi_2}(\Rd))$, and extend the corresponding results from \cite{Cor-Nic}.
We note that our proof contain nontrivial modifications of the technique used in \cite{Cor-Nic}.
In particular, we  use dilation properties of Orlicz spaces given in \cite{Os-Blasco}, and interpolation of Orlicz spaces, cf \cite{Maligranda}.

Let $f_\lambda (x) = f(\lambda x) $, $\lambda >0$, $x \in \Rd$.
The first  estimates for dilations on Wiener amalgam spaces of Orlicz type are given in the next Lemma.

\begin{lemma}\label{lem.2.5.2}
Let $\Phi_1, \Phi_2$ be two Young functions. Then,
\begin{equation}\label{eq.14}
    \|f_\lambda \|_{W(L^{\Phi_1}, L^{\Phi_2})}^{\circ}  \leq \frac{1}{\Phi_1^{-1} (\lambda^{d}) \Phi_2^{-1} (\lambda^{d})} \|f\|_{W(L^{\Phi_1}, L^{\Phi_2})}^{\circ}, \qquad 0<\lambda \leq 1,
\end{equation}
and
\begin{equation}\label{eq.15}
    \|f_\lambda \|_{W(L^{\Phi_1}, L^{\Phi_2})}^{\circ}  \leq \lambda^d \frac{1}{\Phi_1^{-1} (\lambda^{d}) \Phi_2^{-1} (\lambda^{d})} \|f\|_{W(L^{\Phi_1}, L^{\Phi_2})}^{\circ}, \qquad \lambda \geq 1.
\end{equation}
\end{lemma}

\begin{proof}
To compute the norm, we choose the control function $g$ as $g=\chi_{_{\scriptstyle [0,1]^d}}$, the characteristic function of the box $Q=[0,1]^d$. Then, we have
$$\|f_\lambda \|_{W(L^{\Phi_1}, L^{\Phi_2})}^{\circ}=\|F_{_{ \scriptstyle f_\lambda}}^{g} \|_{L^{\Phi_2}}^{\circ}=\big\| \|f_\lambda T_x g\|_{L^{\Phi_1}}^{\circ} \big\|_{L^{\Phi_2}}^{\circ}.$$

We first estimate the norm of the local component $\|\cdot\|_{L^{\Phi_1}}^{\circ}$. We have
\begin{align}\label{eq.1}
\|f_\lambda T_x g\|_{L^{\Phi_1}}^{\circ}&= \inf\Big\{k>0: \int_{\Rd} \Phi_1\big(\frac{|f(\lambda t) g(t-x)|}{k} \big) dt \leq 1\Big\} \nonumber \\
&= \inf\Big\{k>0: \int_{\Rd} \Phi_1\big(\frac{|f(t)g(t/{\lambda} -x)|}{k} \big) \frac{dt}{\lambda^d} \leq 1\Big\} \nonumber \\
&= \inf\Big\{k>0: \int_{\Rd} \Phi_1\big(\frac{|f(t)g_{_{\scriptstyle 1/{\lambda}}}(t-\lambda x)|}{k} \big) \frac{dt}{\lambda^d} \leq 1\Big\}. \end{align}

Let $ 0<\lambda \leq 1$.
By \eqref{eq.dil}, we obtain
\begin{align*}
\|f_\lambda T_x g\|_{L^{\Phi_1}}^{\circ}&=
\inf\Big\{k>0: \int_{\Rd} \Phi_1\big(\frac{|f(t)g(t-\lambda x)|}{k} \big) dt \leq \lambda^d \Big\} \nonumber \\
&= \|f T_{\lambda x} g\|_{L^{\Phi_1}}^{\circ, \lambda} \nonumber\\
&\leq \frac{1}{\Phi_1^{-1}(\lambda^{d})} \|f T_{\lambda x} g\|_{L^{\Phi_1}}^{\circ} \nonumber \\
&= \frac{1}{\Phi_1^{-1}(\lambda^{d})} F_{f}^{g} (\lambda x) =\frac{1}{\Phi_1^{-1}(\lambda^{d})} D_{\lambda} F_{f}^{g} (x),
\end{align*}
where the window function $g$ fulfills $g_{_{\scriptstyle 1/{\lambda}}} \leq g$ for $0< \lambda \leq 1$. By the solidity of $L^{\Phi_2}(\Rd)$ and \eqref{eq.dil} we have
$$\|F_{_{ \scriptstyle f_\lambda}}^{g} \|_{L^{\Phi_2}}^{\circ} \leq \frac{1}{\Phi_1^{-1}(\lambda^{d})} \|D_{\lambda} F_{f}^{g} \|_{L^{\Phi_2}}^{\circ} \leq \frac{1}{\Phi_1^{-1}(\lambda^{d})} \frac{1}{\Phi_2^{-1}(\lambda^{d})} \|F_{f}^{g} \|_{L^{\Phi_2}}^{\circ}$$
which gives
$$\|f_\lambda \|_{W(L^{\Phi_1}, L^{\Phi_2})}^{\circ}  \lesssim \frac{1}{\Phi_1^{-1} (\lambda^{d}) \Phi_2^{-1} (\lambda^{d})} \|f\|_{W(L^{\Phi_1}, L^{\Phi_2})}^{\circ},$$
when  $ 0<\lambda \leq 1$. Next we calculate the norm of local component $\|\cdot\|_{L^{\Phi_1}}^{\circ}$ for $\lambda \geq 1$. By \eqref{eq.1}, we have
\begin{align*}
\|f_\lambda T_x g\|_{L^{\Phi_1}}^{\circ}
& \leq \|f T_{\lambda x} g_{_{\scriptstyle 1/{\lambda}}}\|_{L^{\Phi_1}}^{\circ, \lambda} \nonumber\\
&\leq \frac{1}{\Phi_1^{-1}(\lambda^{d})} \|f T_{\lambda x} g_{_{\scriptstyle 1/{\lambda}}}\|_{L^{\Phi_1}}^{\circ} \nonumber\\
&= \frac{1}{\Phi_1^{-1}(\lambda^{d})} F_{f}^{g_{_{\scriptstyle 1/{\lambda}}}} (\lambda x) \nonumber \\
&= \frac{1}{\Phi_1^{-1}(\lambda^{d})} D_{\lambda} F_{f}^{g_{_{\scriptstyle 1/{\lambda}}}} (x).
\end{align*}

 We proceed by estimating the global component. By the solidity of $L^{\Phi_2}(\Rd)$ and \eqref{eq.dil} we have
$$\|F_{_{ \scriptstyle f_\lambda}}^{g} \|_{L^{\Phi_2}}^{\circ} \leq \frac{1}{\Phi_1^{-1}(\lambda^{d})} \|D_{\lambda} F_{f}^{g_{_{\scriptstyle 1/{\lambda}}}} \|_{L^{\Phi_2}}^{\circ} \leq \frac{1}{\Phi_1^{-1}(\lambda^{d})} \frac{1}{\Phi_2^{-1}(\lambda^{d})} \|F_{f}^{g_{_{\scriptstyle 1/{\lambda}}}} \|_{L^{\Phi_2}}^{\circ}.$$

Thus, we finally obtain
\begin{align*}
 \|f_\lambda \|_{W(L^{\Phi_1}, L^{\Phi_2})}^{\circ}&\leq \frac{1}{\Phi_1^{-1}(\lambda^{d})} \frac{1}{\Phi_2^{-1}(\lambda^{d})} \big\| \|f(t) T_x g_{_{\scriptstyle 1/{\lambda}}}(t)\|_{L_{(t)}^{\Phi_1}}^{\circ} \big\|_{L_{(x)}^{\Phi_2}}^{\circ} \\
 &= \frac{1}{\Phi_1^{-1}(\lambda^{d})} \frac{1}{\Phi_2^{-1}(\lambda^{d})}\big\| \|f(t)  g_{_{\scriptstyle 1/{\lambda}}}(t-x)\|_{L_{(t)}^{\Phi_1}}^{\circ} \big\|_{L_{(x)}^{\Phi_2}}^{\circ} \\
 &\leq  \frac{1}{\Phi_1^{-1}(\lambda^{d})} \frac{1}{\Phi_2^{-1}(\lambda^{d})} \Big\| \Big\|  \sum_{j \in \mathbb{Z} \cap [0,\lambda]^d} f(t) g(t-(x+j))\Big\|_{L_{(t)}^{\Phi_1}}^{\circ} \Big\|_{L_{(x)}^{\Phi_2}}^{\circ} \\
 &= \frac{1}{\Phi_1^{-1}(\lambda^{d})} \frac{1}{\Phi_2^{-1}(\lambda^{d})}\Big\| \sum_{j \in \mathbb{Z} \cap [0,\lambda]^d} F_{f}^{g} (x+j)\Big\|_{L_{(x)}^{\Phi_2}}^{\circ} \\
 &\leq \lambda^d \frac{1}{\Phi_1^{-1}(\lambda^{d})} \frac{1}{\Phi_2^{-1}(\lambda^{d})} \| F_{f}^{g} \|_{L^{\Phi_2}}^{\circ} \\
  &= \lambda^d \frac{1}{\Phi_1^{-1} (\lambda^{d})\Phi_2^{-1} (\lambda^{d})} \|f\|_{W(L^{\Phi_1}, L^{\Phi_2})}^{\circ},
\end{align*}
 where $g_{_{\scriptstyle 1/{\lambda}}}(t) \leq \sum_{j \in \mathbb{Z} \cap [0,\lambda]^d} g(t-j)$ holds for $\lambda \geq 1$. Notice that the sum contains $N_{\lambda}=O(\lambda^d)$ terms (see \cite{Cor-Nic}).
\end{proof}

When restricting to Lebesgue spaces, i.e. when
 $\Phi_1(x)=x^p$ and $\Phi_2(x)=x^q$ for $1 \leq p,q\leq \infty$ (and therefore $\Phi_1^{-1}(\lambda^{d})=\lambda^{d/p}$ and $\Phi_2^{-1}(\lambda^{d})=\lambda^{d/q}$ for $\lambda >0$, \cite{R-R2}) we recover \cite[Lemma 2.3]{Cor-Nic} (see also \cite[Lemma 2.5.2]{Cor-Rod}):

\begin{corollary}
Let $1 \leq p,q \leq \infty$, and let $W(L^p (\Rd), L^q (\Rd))$ be the Wiener amalgam space. Then we have:
\begin{equation*}
    \|f_\lambda \|_{W(L^p, L^q)} \lesssim \lambda^{-d(1/p +1/q)} \|f\|_{W(L^p, L^q)}, \qquad \forall ~~ 0<\lambda \leq 1,
\end{equation*}
and
\begin{equation*}
    \|f_\lambda \|_{W(L^p, L^q)}  \lesssim \lambda^{d(1-1/p -1/q)} \|f\|_{W(L^p, L^q)}, \qquad \forall ~~\lambda \geq 1.
\end{equation*}
\end{corollary}

The estimates given in Lemma \ref{lem.2.5.2} can be improved by using the interpolation arguments in a similar was as it is done in \cite{Cor-Nic}. However, instead of the classical interpolation between Lebesgue type spaces, we should combine the interpolation of Orlicz spaces with the properties of Wiener amalgam spaces.

To that end we first recall the interpolation result for Orlicz spaces
which is given in \cite[Lemma 14.2]{Maligranda} (see \cite{B-S, Calderon} for more details and for
the interpolation notation).

\begin{lemma}\label{lem.inter.orlicz}
    Let $\Phi_0, \Phi_1$ be any Young functions. Then, the function $\Phi$ defined by
    \begin{equation}\label{inter.orlicz}
        \Phi^{-1} (x)=\Phi_0^{-1}(u) \rho\big(\frac{\Phi_1^{-1} (u)}{\Phi_0^{-1}(u)}\big)
    \end{equation}
    is a Young function and $L^{\Phi}(\Rd)= [L^{\Phi_0}(\Rd), L^{\Phi_1}(\Rd)]_{\rho}$, where
     $\rho: [0,\infty) \to [0,\infty)$ is concave, continuous, positive on $(0,\infty)$ and such that
     $$\rho(s) \leq \max\{\frac{s}{t}, 1\}\rho(t), \qquad s,t>0.$$
\end{lemma}
In particular, when $\rho(t)=t^{\theta}$, $0< \theta <1$, we have $\Phi^{-1}=(\Phi_0^{-1})^{1-\theta} (\Phi_1^{-1})^{\theta}$.

The complex interpolation for general Wiener amalgam spaces is given in \cite{Fei2}.
We are interested here in the interpolation result for Wiener amalgam spaces of Orlicz type.
Let $[X, Y]_\rho$ denote the interpolation between certain Orlicz spaces $X$ and $Y$. Then we have the following result.

\begin{lemma}\label{amal.inter}
    Let $B_0, B_1$ be local components of Wiener type Orlicz spaces and $\Phi, \Phi_0, \Phi_1$ be Young functions satisfying the $\Delta_2$ condition, and \eqref{inter.orlicz}. Then, we have
$$[W(B_0, L^{\Phi_0}), W(B_1, L^{\Phi_1})]_\rho=W([B_0 , B_1]_\rho, [L^{\Phi_0}, L^{\Phi_1}]_\rho) =W([B_0 , B_1]_\rho, L^{\Phi})$$
\end{lemma}

Note that the Orlicz space $L^{\Phi}(\Rd)$ has absolutely continuous norm for any Young function $\Phi$ which satisfies the $\Delta_2$ condition. So, Lemma \ref{amal.inter} is a special case of \cite[Theorem 2.2]{Fei2}.

Let $ \mathcal{O}$ be the set of Young functions such that the following conditions hold:
\begin{enumerate}
    \item[i)] $\Phi_s, \Phi_b \in \mathcal{O} $;
    \item[ii)] if $\Phi_1, \Phi_2 \in \mathcal{O} $ then $\Phi_1$ is strictly stronger than  $\Phi_2$ or
    $\Phi_2$ is strictly stronger than  $\Phi_1$. In other words, either $\Phi_1 (x) \leq  \Phi_2 (x)$, $ x \geq 0$, or $\Phi_2 (x) \leq  \Phi_1 (x)$, $ x \geq 0$.
\end{enumerate}

Now we are ready to prove the main result of the paper.

\begin{proposition} \label{main.result}
    Let $\Phi_1, \Phi_2 \in \mathcal{O} $  satisfy the $\Delta_2$ condition and let
    \begin{equation} \label{positive.derivative}
       (\Phi_i)'_{+}(0)>0, \quad \text{and} \quad (\Psi_i)'_{+}(0)>0, ~~i=1, 2.
    \end{equation}
Then we have:
  \begin{equation}\label{eq.5}
    \|f_\lambda \|_{W(L^{\Phi_1}, L^{\Phi_2})}^{\circ}  \lesssim \frac{1}{\max\{\Phi_1^{-1} (\lambda^{d}), \Phi_2^{-1}(\lambda^{d})\}} \|f\|_{W(L^{\Phi_1}, L^{\Phi_2})}^{\circ}, \qquad 0<\lambda \leq 1
\end{equation}
and
\begin{equation}\label{eq.6}
   \|f_\lambda \|_{W(L^{\Phi_1}, L^{\Phi_2})}^{\circ}  \lesssim \frac{1}{\min\{\Phi_1^{-1} (\lambda^{d}), \Phi_2^{-1}(\lambda^{d})\}} \|f\|_{W(L^{\Phi_1}, L^{\Phi_2})}^{\circ}, \qquad \lambda \geq 1.
\end{equation}
\end{proposition}

\begin{proof}
First we note that \eqref{positive.derivative} is needed in order to use Proposition \ref{thm6} and the $\Delta_2$ condition is needed for the complex interpolation.

Let $\Phi_1, \Phi_2$ satisfy the conditions of Proposition \ref{main.result}.  We note that
$\Phi_1$ and $\Phi_2$ give rise to intermediate spaces
between the spaces generated by $\Phi_s$ and $\Phi_b$, so the Young functions
 $\Phi_s$ and $\Phi_b$ play an essential role in the proof.

When  $\Phi_1 =\Phi_2$ by \eqref{eq.dil} we have
$$\|f_\lambda \|_{W(L^{\Phi_1}, L^{\Phi_1})}^{\circ} \asymp \|f_\lambda \|_{L^{\Phi_1}}^{\circ} \leq \frac{1}{\Phi_1^{-1}(\lambda^{d})} \|f\|_{L^{\Phi_1}}^{\circ}.$$

Next we prove \eqref{eq.5} and \eqref{eq.6} for $\Phi_1= \Phi_s$ (cf. \eqref{yf}) and for any Young function $\Phi_2$ with $\Phi_2 \leq \Phi_s$.
Since $\Phi^{-1}(\lambda^{d})=1$ for all $\lambda >0$,
by \eqref{eq.14} we have
\begin{align*}
  \|f_\lambda \|_{W(L^{\Phi_s}, L^{\Phi_2})}^{\circ} &\lesssim  \frac{1}{\Phi_s ^{-1} (\lambda^{d}) \Phi_2^{-1} (\lambda^{d})} \|f\|_{W(L^{\Phi_s }, L^{\Phi_2})}^{\circ} \\
  &= \frac{1}{\Phi_2^{-1} (\lambda^{d})} \|f\|_{W(L^{\Phi_s}, L^{\Phi_2})}^{\circ} =\frac{1}{\max\{\Phi_s ^{-1} (\lambda^{d}), \Phi_2^{-1}(\lambda^{d})\}} \|f\|_{W(L^{\Phi_s}, L^{\Phi_2})}^{\circ},
    \end{align*}
which is \eqref{eq.5}.

On the other hand, \eqref{eq.6} follows by complex interpolation from \eqref{eq.15} with $(\Phi_1, \Phi_2)=(\Phi_s, \Phi_b)$, that is
\begin{align*}
  \|f_\lambda \|_{W(L^{\Phi_s}, L^{\Phi_b})}^{\circ} & \lesssim  \frac{\lambda^{d}}{\Phi_s ^{-1} (\lambda^{d}) \Phi_b^{-1} (\lambda^{d})} \|f\|_{W(L^{\Phi_s}, L^{\Phi_b})}^{\circ}
  = \frac{\lambda^{d}}{\lambda^{d} +1} \|f\|_{W(L^{\Phi_s}, L^{\Phi_b})}^{\circ} \\
  &\leq \|f\|_{W(L^{\Phi_s}, L^{\Phi_b})}^{\circ} \\
  &=\frac{1}{\min\{\Phi_s ^{-1} (\lambda^{d}), \Phi_b^{-1}(\lambda^{d})\}} \|f\|_{W(L^{\Phi_s}, L^{\Phi_b})}^{\circ},
\end{align*}
and since $\Phi_s$ satisfies \eqref{positive.derivative}, trivial estimate
$$\|f_\lambda \|_{W(L^{\Phi_s}, L^{\Phi_s})}^{\circ} \asymp \|f_\lambda \|_{L^{\Phi_s}}^{\circ} =\|f \|_{L^{\Phi_s}}^{\circ} \asymp \|f\|_{W(L^{\Phi_s}, L^{\Phi_s})}^{\circ}.$$
By complex interpolation between the spaces $W(L^{\Phi_s}, L^{\Phi_b})$ and $W(L^{\Phi_s}, L^{\Phi_s})$, we have
\begin{equation}\label{eq.A}
\Phi_2 ^{-1}(u)= \Phi_b^{-1}(u)\rho(\frac{\Phi_s ^{-1}(u)}{\Phi_b^{-1}(u)}) =(u+1)\rho(\frac{1}{u+1}) \leq (u+1)\rho(1),
\end{equation}
which means $\Phi_2 ^{-1}(u) \lesssim \Phi_b^{-1}(u)$, so
\begin{equation}\label{eq.*}
    \Phi_b (u) \lesssim \Phi_2 (u).
\end{equation}
Using the concavity of $\rho$ in \eqref{eq.A}, we obtain
$$\rho(1)=\rho\big((u+1) \frac{1}{u+1}\big) \leq (u+1)\rho(\frac{1}{u+1}) = \Phi_2 ^{-1} (u)$$
which gives $\Phi_s ^{-1} (u) \lesssim \Phi_2 ^{-1} (u)$, so
\begin{equation}\label{eq.**}
    \Phi_2 (u)\lesssim \Phi_s (u).
\end{equation}
Combining \eqref{eq.*} and \eqref{eq.**}, we have
$$\Phi_b (u) \lesssim \Phi_2 (u) \lesssim \Phi_s (u).$$

Thus, the estimates \eqref{eq.5} and \eqref{eq.6} hold for $\Phi_1 = \Psi_s$ and  $\Phi_2 \leq \Phi_s$.
By using the same complex interpolation arguments we conclude that \eqref{eq.5} and \eqref{eq.6} also hold when  $\Phi_2 \leq \Phi_1$, for any Young functions satisfying the conditions of Proposition \ref{main.result}.

Next assume that $\Phi_1 < \Phi_2$. Then $\Psi_2 < \Psi_1$ by \cite[Theorem 2, p.16]{R-R}, so that $\Psi_1^{-1} \lesssim \Psi_2^{-1}$. We will prove \eqref{eq.6}, i.e. the case $\lambda \geq 1$, and note that \eqref{eq.5} can be proved in a similar way.

In the following calculation $ <\cdot, \cdot >$ denotes duality between the corresponding spaces.
The relation \eqref{eq.5} applied to the pair $(\Psi_1, \Psi_2)$, yields
    \begin{align}\label{eq.7}
  \|f_\lambda \|_{W(L^{\Phi_1}, L^{\Phi_2})}^{\circ}&= \sup_{\|g \|_{W(L^{\Psi_1}, L^{\Psi_2})}^{\circ} =1} | < f_{\lambda}, g> |  \nonumber\\
  &= \lambda^{-d} \sup_{\|g \|_{W(L^{\Psi_1}, L^{\Psi_2})}^{\circ} =1} | < f, g_{_{\scriptstyle 1/{\lambda}}}> |  \nonumber\\
  & \leq 4 \lambda^{-d} \|f \|_{W(L^{\Phi_1}, L^{\Phi_2})}^{\circ} \|g_{_{\scriptstyle 1/{\lambda}}} \|_{W(L^{\Psi_1}, L^{\Psi_2})}^{\circ} \nonumber\\
  &\lesssim \lambda^{-d} \frac{1}{ \max\{\Psi_1^{-1} (1/\lambda^{d}),  \Psi_2^{-1} (1/\lambda^{d})\}}  \|f \|_{W(L^{\Phi_1}, L^{\Phi_2})}^{\circ} \|g \|_{W(L^{\Psi_1}, L^{\Psi_2})}^{\circ} \nonumber\\
 &\lesssim  \lambda^{-1} \frac{1}{\Psi_2^{-1} (1/\lambda^{d})}  \|f \|_{W(L^{\Phi_1}, L^{\Phi_2})}^{\circ}
   \end{align}
where it is observed that
$$\|g_{_{\scriptstyle 1/{\lambda}}}\|_{W(L^{\Psi_1}, L^{\Psi_2})}^{\circ} \leq \frac{1}{ \max\{\Psi_1^{-1} (1/\lambda),  \Psi_2^{-1} (1/\lambda)\}} \|g\|_{W(L^{\Psi_1}, L^{\Psi_2})}^{\circ},$$
since $1/\lambda \leq 1$.

By the concavity of $\Psi_2^{-1}$, we have $\frac{1}{\Psi_2^{-1}(1/\lambda^{d})} \leq \Psi_2^{-1}(\lambda^{d})$ for $\lambda \geq 1$. From \eqref{eq.7}, we obtain
\begin{equation}\label{eq.8}
 \|f_\lambda \|_{W(L^{\Phi_1}, L^{\Phi_2})}^{\circ} \lesssim \lambda^{-d} \Psi^{-1}(\lambda^{d})  \|f\|_{W(L^{\Phi_1}, L^{\Phi_2})}^{\circ}.
\end{equation}

By \cite[Proposition 1, p.13]{R-R} it follows that $u \leq \Phi^{-1}(u) \Psi^{-1}(u) < 2 u$ for any $u >0$ and any complementary Young pair $(\Phi, \Psi)$. Hence we have
\begin{align*}
 \|f_\lambda \|_{W(L^{\Phi_1}, L^{\Phi_2})}^{\circ} &\lesssim 2\lambda^{-d} \lambda \frac{1}{\Phi_2^{-1}(\lambda^{d})}  \|f\|_{W(L^{\Phi_1}, L^{\Phi_2})}^{\circ}  \\  &\lesssim \frac{1}{\Phi_2^{-1}(\lambda^{d})}  \|f\|_{W(L^{\Phi_1}, L^{\Phi_2})}^{\circ} \\
 &= \frac{1}{\min\{\Phi_1^{-1}(\lambda^{d}), \Phi_2^{-1}(\lambda^{d})\}} \|f\|_{W(L^{\Phi_1}, L^{\Phi_2})}^{\circ}
\end{align*}
which is \eqref{eq.6}. As we noted, the case $0< \lambda \leq 1$ can be done in a similar way, so we omit the proof.

Finally, it remains to prove  \eqref{eq.5} and \eqref{eq.6} when $\Phi_1 =\Phi_b$ and
$\Phi_b < \Phi_2$. We again use the complex interpolation between
$(\Phi_1, \Phi_2)=(\Phi_b , \Phi_b)$ and $(\Phi_1, \Phi_2)=(\Phi_b , \Phi)$ as follows.

By \eqref{eq.14} with $(\Phi_1, \Phi_2)=(\Phi_b , \Phi_b)$, we have
\begin{equation*}
  \|f_\lambda \|_{W(L^{\Phi_b}, L^{\Phi_b})}^{\circ} \asymp  \|f_\lambda \|_{L^{\Phi_b}}^{\circ} \leq \frac{1}{\Phi_b^{-1} (\lambda^{d})} \|f\|_{L^{\Phi_2}}^{\circ} \asymp  \frac{1}{\Phi_b^{-1} (\lambda^{d})} \|f\|_{W(L^{\Phi_b}, L^{\Phi_b})}^{\circ}
    \end{equation*}
which is \eqref{eq.5}.

On the other hand, by \eqref{eq.15} with $(\Phi_1, \Phi_2)=(\Phi_b , \Phi)$, we have
\begin{align*}
  \|f_\lambda \|_{W(L^{\Phi_b}, L^{\Phi})}^{\circ} &\leq \frac{\lambda^{d}}{\Phi_b^{-1} (\lambda^{d}) \Phi^{-1} (\lambda^{d})} \|f\|_{W(L^{\Phi_b}, L^{\Phi})}^{\circ} \\
  &= \frac{\lambda^{d}}{\lambda^{d} +1} \|f\|_{W(L^{\Phi_b}, L^{\Phi})}^{\circ} \\
  &\leq \|f\|_{W(L^{\Phi_b}, L^{\Phi})}^{\circ} \\
  &= \frac{1}{\min\{\Phi_b^{-1} (\lambda^{d}), \Phi^{-1} (\lambda^{d})\}} \|f\|_{W(L^{\Phi_b}, L^{\Phi})}^{\circ},
    \end{align*}
which is  \eqref{eq.6}, and the proof is completed.
\end{proof}

Note that when we take $\Phi_1(x)=x^p$ and $\Phi_2(x)=x^q$ for $1 \leq p,q\leq \infty$ in Proposition \ref{main.result}, we have
$$ \|f_\lambda \|_{W(L^p, L^q)}\lesssim \frac{1}{\max\{\lambda^{d/p}, \lambda^{d/q}\}} \|f\|_{W(L^p, L^q)}$$
for $0< \lambda <1$. In this case, we obtain a sharper estimate than \cite[Proposition 2.2]{Cor-Nic} and \cite[Proposition 2.5.1]{Cor-Rod}.

For the sharpness of Proposition \ref{main.result}
we recall the result of Cordero and Nicola related to the  Wiener amalgam spaces of Lebesgue type.
\begin{proposition}(\cite[Proposition 2.2]{Cor-Nic} and \cite[Proposition 2.5.1]{Cor-Rod}) For $1 \leq p,q\leq \infty$,
$$
    \|f_\lambda \|_{W(L^p, L^q)} \lesssim \lambda^{-d\max\{1/p, 1/q\}} \|f\|_{W(L^p, L^q)} \quad , \quad \forall ~~0<\lambda \leq 1$$
    and
    $$\|f_\lambda \|_{W(L^p, L^q)} \lesssim \lambda^{-d\min\{1/p, 1/q\}} \|f\|_{W(L^p, L^q)}\quad , \quad \forall ~~\lambda \geq 1.$$

\end{proposition}

\section{An application of Wiener amalgam spaces of Orlicz type} \label{sec5}

In this section we recall an application of the amalgam space $W(C_0 (\Rd), L^1 (\Rd))$ to  the "Amalgam Balian-Low Theorem" due to C. Heil,
\cite{Heil2}. As we  mentioned in the introduction, Proposition \ref{B.L.1} is an extension of  \cite[Proposition 11.9.5]{Heil}.
The Wiener amalgam space $W(C_0(\Rd), L^1(\Rd))$ plays a central role in the Amalgam Balian-Low Theorem. This space consists of functions which are locally in $C_0$ and globally in $L^1 (\Rd)$, and is called the Wiener algebra. It can be shown that
$$
W(C_0(\Rd), L^1(\Rd))= C_0(\Rd) \cap W(L^{\infty}(\Rd), L^1(\Rd)).
$$

We first recall some notation from Gabor analysis, and
refer to e.g. \cite{Groc, GrHeOk, Heil1, Heil, Heil2} for proofs and details.

Given $g \in L^2(\Rd)$ and $a,b >0$, we consider the time-frequency shifts of $g$ as
\begin{equation}\label{Gabor}
    g_{mn}(x)= e^{2 \pi mbx} g(x-na) = M_{mb} T_{na} g(x), \quad x \in \Rd, \quad m,n \in \mathbb{Z}^d.
\end{equation}
Then the set of time-frequency shifts
$$ \mathcal{G} (g,a,b) =
\{g_{mn}\}_{m,n \in \mathbb{Z}^d} $$
is called a {\em Gabor system}. Recall (cf. \cite{Groc}), a Gabor system is a {\em  Gabor frame} for $L^2(\Rd)$ if
there exist positive constants $A,B>0$ such that
$$
A \| f \| \leq \sum_{m,n}
| \langle f, g_{mn} \rangle |^2
\leq B \| f\|, \qquad \forall f \in L^2(\Rd).
$$
Here $ \langle \cdot, \cdot \rangle $ is the scalar product, and $ \| \cdot \|$ is the norm
in $L^2(\Rd)$.

Next we consider orthonormal bases and Riesz bases in the Hilbert space  $L^2(\Rd)$. A Riesz basis is the image of an orthonormal basis under an invertible mapping, and thus any  orthonormal basis is a Riesz basis. It can be shown that a Gabor system $\{g_{mn}\}_{m,n \in \mathbb{Z}^d}$ can only be a frame for $L^2(\Rd)$ when $ab \leq 1$, and can only be a Riesz basis for $L^2(\Rd)$ when $ab=1$. We are here mostly interested in Gabor systems $\{g_{mn}\}_{m,n \in \mathbb{Z}^d}$ of the form in \eqref{Gabor} which are Riesz bases for $L^2(\Rd)$. Then $ab=1$, and by, a change of variables, it is enough to consider the case $a=b=1$. Therefore, we set $a=b=1$ for the remainder of this section.

To prove the Amalgam Balian-Low Theorem, define the Zak transform of a function $g \in L^2(\Rd)$ to be the function $Zg$ with domain $\mathbb{R}^2$ defined by
$$Zg(t,w)= \sum_{k \in \mathbb{Z}^d} g(t+k)e^{2\pi ikw}, \qquad (t,w) \in \mathbb{R}^{2d}.$$

The following lemmas summarize some of the basic facts about Zak transform, see e.g. \cite{Heil2}.

\begin{lemma}\label{lem.11.9.3}
Let $Q$ be any closed unit square in $\mathbb{R}^{2d}$ and $g \in L^2(\Rd)$.
\begin{enumerate}
    \item[i)] The series defining $Zg$ converges in the norm of $L^2(Q)$.

    \item[ii)] $Z$ is an unitary mapping of $L^2(\Rd)$ onto $L^2(Q)$.

    \item[iii)] $Zg$ satisfies the quasi periodicity relations, for $(t,w) \in \mathbb{R}^{2d}$,
    $$Zg (t+1, w)= e^{-2\pi iw}Zg (t,w), \quad Zg (t,w+1)= e^{-2\pi iw}Zg (t,w).$$

    \item[iv)] If $Zg$ is continuous on $\mathbb{R}^{2d}$, then $Zg$ has a zero in $Q$.

    \item[v)] With $a=b=1$, we have
    $$Z(g_{mn})(x,w)= e^{2\pi imx}e^{2\pi inw} Zg(t,w), \quad m,n \in \mathbb{Z}^{d}.$$
\end{enumerate}
\end{lemma}

\begin{lemma}\label{lem.11.9.4}
Let $g \in L^2(\Rd)$ be fixed, and set $a=b=1$.
\begin{enumerate}
    \item[i)] $\{g_{mn}\}_{m,n \in \mathbb{Z}^d}$ is an orthonormal basis for $L^2(\Rd)$ if and only if $|Zg(x,w)|=1$ a.e.

    \item[ii)] $\{g_{mn}\}_{m,n \in \mathbb{Z}^d}$ is a Riesz basis for $L^2(\Rd)$ if and only if there exist $A$, $B$ such that $0<A \leq |Zg(x,w)| \leq B < \infty$ a.e.
\end{enumerate}
\end{lemma}

Now, by adopting the technique from \cite{Heil}, we give the following proposition for Orlicz amalgam spaces.

\begin{proposition}\label{B.L.1}
    Let $\Phi$ be a Young function and $Q$ be a unit cube in $\mathbb{R}^{2d}$. Then the Zak transform is a continuous, linear, injective map of $W(L^\Phi(\Rd), L^1 (\Rd))$ into $L^\Phi(Q)$.
\end{proposition}
\begin{proof}
    Fix $f \in W(L^\Phi(\Rd), L^1 (\Rd))$. For $k \in \mathbb{Z}^d$, define $F_k (t, w)= f(t+k)e^{2\pi ikw}$. Since $\|F_k \|_{L^\Phi } =\|f \chi_{_{\scriptstyle Q+k}} \|_{L^{\Phi}}$, we have $F_k \in L^\Phi(Q)$. Then,
    $$\|Zf \|_{L^\Phi (Q)} = \Big\|\sum_{k \in \mathbb{Z}^d} F_k\Big\|_{L^\Phi} \leq \sum_{k \in \mathbb{Z}^d} \|F_k \|_{L^\Phi} =\sum_{k \in \mathbb{Z}^d} \|f \chi_{_{\scriptstyle Q+k}} \|_{L^{\Phi}} =\|f\|_{W(L^\Phi, L^1)} <\infty.$$
 Hence the series $Zf=\sum_{k \in \mathbb{Z}^d}F_k$ converges absolutely in $L^\Phi(Q)$ and $Z$ is a continuous mapping of $W(L^\Phi(\Rd), L^1(\Rd))$ into $L^\Phi(Q)$.

Now, we show that it is an injective mapping. For the compact subset $k+Q$, we have $L^\Phi(k+Q) \subseteq L^1(k+Q)$. Since $L^1 (\Rd)$ is a solid space, we obtain
$$W(L^\Phi(\Rd), L^1(\Rd))\subseteq W(L^1(\Rd), L^1 (\Rd))=L^1(\Rd).$$
By \cite[Proposition 7.4.1]{Heil1}, $Z$ is injective.
\end{proof}

In Proposition \ref{B.L.1}, when we take $\Phi(x)=x^p$ for $1 \leq p <\infty$, we obtain the following result for $L^p$ spaces, see \cite[Proposition 7.5.1]{Heil1} or \cite[Proposition 11.9.5]{Heil}.

\begin{corollary}
    Let $1 \leq p <\infty$ be given and let $Q$ be a unit cube in $\mathbb{R}^{2d}$. Then the Zak transform is a continuous, linear, injective map of $W(L^p(\Rd), L^1 (\Rd))$ into $L^p(Q)$.
\end{corollary}

Note that since $W(C_0, L^1)= C_0 \cap W(L^{\infty}(\Rd), L^1(\Rd))$ and $L^{\Phi}(\Rd)$ is $L^{\infty}(\Rd)$ in the case
\begin{equation*}
\Phi(x)=
\begin{cases}
0, & 0 \leq x \leq 1,\\
\infty,& x>1,
\end{cases}
\end{equation*}
we obtain the following well known result which is given in \cite{Heil1, Heil, Heil2}.

\begin{corollary}\label{cor.5.4}
    If $f \in W(C_0 (\Rd), L^1 (\Rd))$, then $Zf$ is continuous on $\mathbb{R}^{2d}$.
\end{corollary}

Thus the following theorem can be viewed as a special case of Orlicz space results.

\begin{theorem}\textbf{(Amalgam Balian-Low Theorem)}~~
Fix $g \in L^2(\Rd)$ and set $a=b=1$. If $g_{mn} \in \mathbb{Z}^d$ is a Riesz basis for $L^2(\Rd)$, then $g, \hat{g} \in W(C_0 (\Rd), L^1 (\Rd))$.
\end{theorem}

For the proof of the Amalgam Balian-Low Theorem we refer to \cite{Heil2}, and give only a sketch here.
If $g \in W(C_0 (\Rd), L^1 (\Rd))$, then $Zg$ is continuous by Proposition \ref{B.L.1}
(or Corollary \ref{cor.5.4}). Then $Zg$ has a zero by Lemma \ref{lem.11.9.3}, and  by Lemma \ref{lem.11.9.4} the Gabor system $\{g_{mn}\}_{m,n \in \mathbb{Z}^d}$ cannot be a Riesz basis for $L^2(\Rd)$. This argument also applies to $\hat{g}$ because by applying the Fourier transform we see that $\{{\hat{g}}_{mn}\}_{m,n \in \mathbb{Z}^d}$ is a Riesz basis for $L^2(\Rd)$ if and only if $\{g_{mn}\}_{m,n \in \mathbb{Z}^d}$ is a Riesz basis for $L^2(\Rd)$.



\section{Conflict of interest}

On behalf of all authors, the corresponding author states that there is no conflict of interest.

\section*{Acknowledgements}

The first author thanks Tinçel Kültür Vakfı and University of Novi Sad for the support during Spring 2023 her stay at University of Novi Sad. This study was funded by Scientific Research Projects Coordination Unit of \.{I}stanbul University, project number FBA-2023-398-40.

N. Teofanov was  supported by the Science Fund of the Republic of
Serbia, $\#$GRANT No 2727, \emph{Global and local analysis of operators and
distributions} - GOALS, and acknowledges the financial support of the Ministry of Science, Technological Development and Innovation of the Republic of Serbia (Grants No. 451-03-66/2024-03/ 200125 $\& $ 451-03-65/2024-03/200125).

The authors are grateful to anonymous referees for
careful reading of the manuscript and for their valuable comments.


\end{document}